\documentclass{article}

\usepackage{amsmath, amssymb, amsthm}
\usepackage{url}
\usepackage{prettyref}
\usepackage{fullpage,enumerate}
\usepackage{boxedminipage}

\usepackage[colorlinks=true,citecolor=blue,linkcolor=red,urlcolor=green]{hyperref}

\newcommand{\bs}{\backslash}
\newcommand{\comment}[1]{}
\newcommand{\ignore}[1]{}
\renewcommand{\P}{\ensuremath{\mathcal{P}}}

\def\Ev{\ensuremath{\mathfrak{E}}}
\def\F{\ensuremath{\mathcal{F}}}
\def\E{\ensuremath{\mathcal{E}}}

\def\reals{\ensuremath{\mathbb{R}}}
\def\r{r}
\def\R{R}
\def\MR{\mathrm{R}}
\def\MB{\mathrm{B}}
\def\PP{\ensuremath{\mathsf{P}}}
\def\NP{\ensuremath{\mathsf{NP}}}

\def\PTT{\ensuremath{\mathrm{PTT}}}

\def\Tr{{\textsc{PathsInTrees}}}
\def\Tri{\blacktriangle}

\def\Hgr{{\textsc{Hyp}}}
\def\deg{\ensuremath{\mathrm{deg}}}
\def\Chi{\ensuremath{\mathsf{p}}}
\def\CC{\ensuremath{\overline{\Chi}}}
\def\ChiDual{\ensuremath{\mathsf{cd}}}
\def\CCDual{\ensuremath{\overline{\mathsf{cd}}}}
\def\cd{\ChiDual}
\newtheorem{theorem}{Theorem}
\newrefformat{theorem}{Theorem~\ref{#1}}
\newrefformat{appendix}{Appendix~\ref{#1}}
\newrefformat{eq}{Equation~\eqref{#1}}
\newrefformat{chap}{Chapter~\ref{#1}}
\newrefformat{fig}{Figure~\ref{#1}}
\def\ythm[#1][#2][#3]{\newtheorem{#2}[theorem]{#3} \newrefformat{#2}{#3 \ref{#11}}}
\def\xthm[#1][#2][#3]{\newrefformat{#2}{#3 \ref{#11}}}
\ythm[#][definition][Definition]
\ythm[#][claim][Claim]
\ythm[#][conjecture][Conjecture]
\ythm[#][proposition][Proposition]
\ythm[#][lemma][Lemma]
\ythm[#][fact][Fact]
\ythm[#][corollary][Corollary]

\begin{document}

\title{Cover-Decomposition and Polychromatic Numbers\thanks{A preliminary version of this paper appeared as \cite{BPRS11}.}}
\date{\today}
\author{B{\'e}la Bollob{\'a}s\thanks{University of Memphis, USA and University of Cambridge, UK; supported in part by NSF grants CNS-0721983, CCF-0728928, DMS-0906634 and CCR-0225610, and ARO grant W911NF-06-1-0076.} \and David Pritchard\thanks{University of Waterloo, Canada. Work performed while at \'Ecole Polytechnique F\'ed\'erale de Lausanne, Switzerland and supported by an NSERC Post-Doctoral Fellowship.}\and Thomas Rothvo\ss\thanks{MIT, Cambridge, USA; supported by the Alexander von Humboldt Foundation within the Feodor Lynen program, by ONR grant N00014-11-1-0053, and by NSF contract CCF-0829878.}\and Alex Scott\thanks{Mathematical Institute, University of Oxford, UK.}}

\maketitle
\begin{abstract}
A colouring of a hypergraph's vertices is \emph{polychromatic} if every hyperedge contains at least one vertex of each colour; the \emph{polychromatic number} is the maximum number of colours in such a colouring. Its dual, the \emph{cover-decomposition number}, is the maximum number of disjoint hyperedge-covers. In geometric hypergraphs, there is extensive work on lower-bounding these numbers in terms of their trivial upper bounds (minimum hyperedge size and degree); our goal here is to broaden the study beyond geometric settings. We obtain algorithms yielding near-tight bounds for three families of hypergraphs: bounded hyperedge size, paths in trees, and bounded VC-dimension. This reveals that discrepancy theory and iterated linear program relaxation are useful for cover-decomposition. Finally, we discuss the generalization of cover-decomposition to sensor cover.
\end{abstract}

%


\section{Introduction}
In a set system on vertex set $V$, a subsystem is a \emph{set cover} if each vertex of $V$ appears in at least 1 set of the subsystem. Suppose that in the whole system, each vertex appears in at least $\delta$ sets, for some large $\delta$; does it follow that we can partition the system into 2 subsystems, such that each subsystem is a set cover?

Many natural families of set systems admit a universal constant $\delta$ for which this question has an affirmative answer. Such families are called \emph{cover-decomposable}. But the family of \emph{all} set systems is not cover-decomposable, as the following example shows. For any positive integer $k$, consider a set system which has $2k-1$ sets, and where every $k$ sets contain one mutually common vertex not contained by the other $k-1$ sets. This system satisfies the hypothesis of the question for $\delta = k$. But every set cover has $\ge k$ sets, and since there are only $2k-1$ sets in total, no partition into two set covers is possible. This example above shows that some sort of restriction on the family is necessary to ensure cover-decomposability.

One positive example of cover-decomposition arises if every set has size 2: such hypergraphs are simply graphs. They are cover-decomposable with $\delta = 3$: any graph with minimum degree 3 can have its edges partitioned into two edge covers. More generally, Gupta~\cite{Gupta78} showed (see also \cite{ABB+09,Andersen79}) that we can partition the edges of any multigraph into $\lfloor \frac{3 \delta + 1}{4} \rfloor$ edge covers. This bound is tight, even for 3-vertex multigraphs. (In \emph{simple} graphs the optimal bound~\cite{Gupta73} is $\delta-1$.)

Set systems in many geometric settings have been studied with respect to cover-decompos\-abil\-ity; many positive and negative examples are known and there is no easy way to distinguish one from the other. In the affirmative case, as with Gupta's theorem, the next natural problem is to find for each $t \ge 2$ the smallest $\delta(t)$ such that when each vertex appears in at least $\delta(t)$ sets, a partition into $t$ set covers is possible.
The goal of this paper is to extend the study of cover-decomposition beyond geometric settings. 

\subsection{Terminology and Notation}
A \emph{set system} or \emph{hypergraph} $H = (V, \E)$ consists of a ground set $V$ of vertices, together with a collection $\E$ of hyperedges, where each hyperedge $E \in \E$ is a subset of $V$.
We will sometimes call hyperedges just \emph{edges} or \emph{sets}. We permit $\E$ to contain multiple copies of the same hyperedge (e.g.~to allow us to define ``duals" and ``shrinking" later), and we also allow hyperedges of cardinality 0 or 1. We only consider hypergraphs that are finite. Note, in some geometric cases, infinite cover-decomposability problems can be reduced to finite ones. We refer the reader to \cite{Pa10} including the distinction between plane- and total-cover-decomposability. Additional work on the infinite version appears in~\cite{EMS09}.

To \emph{shrink} a hyperedge $E$ in a hypergraph means to replace it with some $E' \subseteq E$. This operation is useful in several places.

A \emph{polychromatic $k$-colouring} of a hypergraph is a function from $V$ to a set of $k$ colours so that for every edge, its image contains all colours. Equivalently, the colour classes partition $V$ into sets which each meet every edge, so-called \emph{vertex covers}/\emph{transversals}. The maximum number of colours in a polychromatic colouring of $H$ is called its \emph{polychromatic number}, which we denote by $\Chi(H)$.

A \emph{cover $k$-decomposition} of a hypergraph is a partition of $\E$ into $k$ subfamilies $\E = \biguplus_{i=1}^k \{\E_i\}$ such that each $\bigcup_{E \in \E_i} E = V$. In other words, each $\E_i$ must be a set cover. The maximum $k$ for which the hypergraph $H$ admits a cover $k$-decomposition is called its \emph{cover-decomposition number}, which we denote by $\ChiDual(H)$.

The \emph{dual} $H^*$ of a hypergraph $H$ is another hypergraph such that the vertex set of $H^*$ corresponds to the edge set of $H$, and vice-versa, with incidences preserved. Thus the vertex-edge incidence matrices for $H$ and $H^*$ are transposes of one
another. E.g., the standard notation for the example in the introduction is $\tbinom{[2k-1]}{k}^*$.
From the definitions it is easy to see that the polychromatic and cover-decomposition numbers are dual to one another,
$$\ChiDual(H) = \Chi(H^*).$$

The \emph{degree} of a vertex $v$ in a hypergraph is the number of hyperedges containing $v$; it is $d$-\emph{regular} if all vertices have degree $d$. We denote the minimum degree by $\delta$, and the maximum degree by $\Delta$. We denote the minimum size of any hyperedge by $r$, and the maximum size of any hyperedge by $R$. Note that $\Delta(H) = R(H^*)$ and $\delta(H) = r(H^*)$. It is trivial to see that $\Chi \le r$ in any hypergraph and dually that $\ChiDual \le \delta$. So the cover-decomposability question asks if there is a converse to this trivial bound: if $\delta$ is large enough, does $\ChiDual$ also grow? To write this concisely, for a family $\F$ of hypergraphs, let its extremal cover-decomposition function $\CCDual(\F, \delta)$ be
$$\CCDual(\F, \delta) := \min \{\ChiDual(H) \mid H \in \F; ~ \forall v \in V(H): \textrm{degree}(v) \ge \delta\},$$
i.e.~$\CCDual(\F, \delta)$ is the best possible lower bound for $\ChiDual$ among hypergraphs in $\F$ with min-degree $\ge \delta$.
So to say that $\F$ is cover-decomposable means that $\CCDual(\F, \delta)>1$ for some constant $\delta$.
We also dually define
$$\CC(\F, \r) := \min \{\Chi(H) \mid H \in \F; ~ \forall E \in \E(H): |E| \ge \r\}.$$
In the rest of the paper we focus on computing the functions $\CCDual$ and $\CC$.

We sometimes write $\Chi$ for $\Chi(H)$ when $H$ is clear from context, and likewise with $\CC$. We write $\CCDual(\delta)$ or $\CC(r)$ when the family $\F$ is clear from context. For example, Gupta's theorem (together with the tight example) says that in graphs, $\CCDual(\delta) = \lfloor \frac{3 \delta + 1}{4} \rfloor$.

\subsection{Results}
In \prettyref{sec:bdsize} we generalize Gupta's theorem to hypergraphs of bounded edge size. Let $\Hgr(R)$ denote the family of hypergraphs with all edges of size at most $\R$.

\begin{theorem}\label{theorem:1}
For all $R, \delta$ we have $\CCDual(\Hgr(R), \delta) \ge \delta/(\ln R+O(\ln \ln R))$.
\end{theorem}

\noindent Before proving \prettyref{theorem:1}, we give a simpler proof of a result that is weaker by a constant factor. This proof uses the Lov\'asz Local Lemma (LLL) and the Beck-Fiala theorem~\cite{BF81}. The Beck-Fiala theorem says that every hypergraph's \emph{discrepancy} (defined later) is less than $2\Delta$. To get the strong version of \prettyref{theorem:1} we employ another discrepancy upper bound of $2\sqrt{R \ln (R\Delta)}$, obtained using the LLL and Chernoff bounds (\prettyref{proposition:chernoff}). Next we show that \prettyref{theorem:1} is always tight up to a constant factor, and that when $R$ is sub-exponential in $\delta$ this constant factor tends to 1.
\begin{theorem}\label{theorem:bounds}
\begin{enumerate}[(i)]\item For all $R \ge 2, \delta \ge 1$ we have $\CCDual(\Hgr(R), \delta) \le \max\{1, O(\delta/\ln R)\}.$ \item For any sequence $R, \delta \to \infty$ with $\delta = \omega(\ln R)$ we have $\CCDual(\Hgr(R), \delta) \le (1+o(1))\delta/\ln(R).$\end{enumerate}
\end{theorem}

\noindent Here (i) uses an explicit construction while (ii) uses the probabilistic method. 
By plugging \prettyref{theorem:1} into an approach of \cite{ABB+09}, one obtains a good bound on the cover-decomposition number of \emph{sparse} hypergraphs.
\begin{corollary}\label{corollary:sparse}
Suppose for some fixed $\alpha, \beta$ that $H = (V, \E)$ satisfies, for all $V' \subseteq V$ and $\E' \subseteq \E$, that the number of incidences between $V'$ and $\E'$ is at most $\alpha|V'|+\beta|\E'|$. Then $\ChiDual(H) \ge \frac{\delta(H)-\alpha}{\ln \beta+O (\ln \ln \beta)}$.
\end{corollary}

\noindent Note that duality yields a similar bound on the polychromatic number.

In \prettyref{sec:structured} we consider the following family of hypergraphs: the ground set is the edge set of an undirected tree, and each hyperedge must correspond to the edges lying in some path in the tree. We show that such systems are cover-decomposable:

\newcounter{thm4}
\setcounter{thm4}{\value{theorem}}

\begin{theorem}\label{theorem:2}
For hypergraphs defined by edges of paths in trees, $\CCDual(\delta) = \Omega(\delta).$
\end{theorem}

\noindent
To prove \prettyref{theorem:2} we exploit iterated LP relaxation, using
an extreme point structure theorem for paths in trees from \cite{KPP08}. The approach also generalizes to path-in-tree systems where the ground set consists of the arcs or the nodes of the tree.
We also determine the extremal polychromatic number for such systems:

\newcounter{thm5}
\setcounter{thm5}{\value{theorem}}

\begin{theorem}\label{theorem:3}
For hypergraphs defined by edges of paths in trees, $\CC(r) = \lceil r/2 \rceil.$
\end{theorem}

\noindent This contrasts with a construction of Pach, Tardos and T\'oth~\cite{PTT09} (described in \prettyref{sec:vc}): if we also allow hyperedges consisting of sets of ``siblings," then $\CC(r) = 1$ for all $r$.

\renewcommand{\thefootnote}{\fnsymbol{footnote}}

The \emph{VC-dimension} is a prominent measure of set system complexity used frequently in geometry: it is the maximum cardinality of any $S \subseteq V$ such that $\{S \cap E \mid E \in \E\} = \mathbf{2}^S$. It is natural to ask what role, if any, the VC-dimension plays in cover-decomposability. In the most restricted case, we get:

\newcounter{thm6}
\setcounter{thm6}{\value{theorem}}

\begin{theorem}\label{theorem:4}
For the family of hypergraphs with VC-dimension 1, $\CC(r) = \lceil r/2 \rceil$ and $\CCDual(\delta) = \lceil \delta/2 \rceil$.
\end{theorem}

\noindent By duality, the same holds for the family of hypergraphs whose duals have VC-dimension 1. Roughly, these results hold because VC-dimension 1 implies a strong structural property. Moreover, we observe that a dimension upper bound of 2 does not yield cover-decomposability:

\newcounter{thm7}
\setcounter{thm7}{\value{theorem}}

\begin{theorem}\label{theorem:5}
For the family of hypergraphs $\{H \mid \textrm{VC-dim}(H), \textrm{VC-dim}(H^*) \le 2\}$, we have $\CC(r) = 1$ for all $r$ and $\CCDual(\delta) = 1$ for all $\delta$.
\end{theorem}

\noindent To prove this, we show that the construction of \cite{PTT09} has primal and dual VC-dimension at most 2.

In the last section, we discuss the generalization of cover-decomposition to cover-scheduling. This is inspired by results from \cite{GK09}, where a proof of $\cd = \Omega(\delta)$ for convex polygons in the plane is generalized to a scheduling setting while losing only a constant factor. We do not know if \prettyref{theorem:1} has an analogue within a constant factor; we prove a weaker version with linear dependence on $R$.

All of our lower bounds on $\CC$ and $\CCDual$ can be implemented as polynomial-time algorithms. In the case of \prettyref{theorem:1} this relies on the constructive LLL framework of Moser-Tardos~\cite{MT10}. In the tree setting (\prettyref{theorem:2}) this relies on a linear programming subroutine. Note: since we also have the trivial bounds $\Chi \le r$ and $\ChiDual \le \delta,$ these give approximation algorithms for $\Chi$ and $\ChiDual$, e.g.~\prettyref{theorem:1} gives a $(\ln R + O(\ln \ln R))$-approximation for $\ChiDual$. 

\subsection{Related Work}\label{sec:closing}
One practical motive to study cover-decomposition is that the hypergraph can model a collection of sensors~\cite{BEJVY07,GK09}, with each $E \in \E$ corresponding to a sensor which can monitor the subset $E \subset V$ of vertices; then monitoring all of $V$ takes a set cover, and $\CCDual$ is the maximum ``coverage" of $V$ possible if each sensor can only be turned on for a single time unit or monitor a single frequency. A motivation from theory is that if $\CCDual(\delta) = \Omega(\delta)$ holds for a family closed under vertex deletion, then the size of a \emph{dual $\epsilon$-net} is bounded by $O(1/\epsilon)$~\cite{PT11}.

A hypergraph is said to be \emph{weakly $k$-colourable} if we can $k$-colour its vertex set so that no edge is monochromatic. Weak 2-colourability is also known as \emph{Property B}, and these notions coincide with the property $\Chi \ge 2$. However, weak $k$-colourability does not imply $\Chi \ge k$ in general. The \emph{discrepancy} of a hypergraph is the minimum $D$ so that we can colour its vertices by $\pm 1$ so that the sum of the values within every edge is between $-D$ and $D.$ This is related to Property B: if a hypergraph's discrepancy is strictly less than $\delta,$ then the same colouring establishes that it has Property B. We will deal with discrepancy in the dual setting.

For a graph $G = (V, E)$, the \emph{(closed) neighbourhood hypergraph} $\mathcal{N}(H)$ is defined to be a hypergraph on ground set $V$, with one hyperedge $\{v\} \cup \{u \mid \{u, v\} \in E\}$ for each $v \in V$. Then $\Chi(\mathcal{N}(G))=\ChiDual(\mathcal{N}(G))$ equals the \emph{domatic number} of $G$, i.e.~the maximum number of disjoint dominating sets. The paper of Feige, Halld\'orsson, Kortsarz \& Srinivasan~\cite{FHKS02} obtains upper bounds for the domatic number and their bounds are essentially the same as what we get by applying \prettyref{theorem:1} to the special case of neighbourhood hypergraphs; compared to our methods they use the LLL but not discrepancy or iterated LP relaxation. They give a hardness-of-approximation result which implies that \prettyref{theorem:1} is tight with respect to the approximation factor, namely for all $\epsilon>0$, it is hard to approximate $\ChiDual$ within a factor better than $(1-\epsilon)\ln R$, under reasonable complexity assumptions. A simpler, weaker hardness result is that you cannot approximate $\ChiDual$ of graphs better than 3/2, since it is \NP-hard to check if a cubic graph has 3 disjoint perfect matchings~\cite{Holyer81}.
A generalization of results in \cite{FHKS02} to packing polymatroid bases was given in \cite{CCV09}; this implies a weak version of \prettyref{theorem:1} where the $\ln R$ term is replaced by $\ln |V|$.

Given a plane graph, define a hypergraph whose vertices are the graph's vertices, and whose hyperedges are its faces. In \cite{ABB+09} it was shown to satisfy the hypothesis of \prettyref{corollary:sparse} for $\alpha=\beta=2$. Their method, which we re-use in the proof of \prettyref{corollary:sparse}, gives $\CC(\delta) \le \lfloor (3\delta-5)/4\rfloor$ from Gupta's theorem.

A notable progenitor in geometric literature on cover-decomposition is the following question of Pach~\cite{Pach80}. Take a convex set $A \subset \reals^2$.
Let $\reals^2|\textsc{Translates}(A)$ denote the family of hypergraphs where the ground set $V$ is a finite subset of $\reals^2$, and each hyperedge is the intersection of $V$ with some translate of $A$. Pach asked if such systems are cover-decomposable, and this question is still open.
A state-of-the-art partial answer is due to Gibson \& Varadarajan~\cite{GK09}, who prove that $\CC(\reals^2|\textsc{Translates}(A), \delta) = \Omega(\delta)$ when $A$ is an open convex polygon; prior work includes~\cite{P86,TT07,PT07,A08,PT10}.
There is an unpublished proof~\cite{MLP87} that $\CC(\reals^2|\textsc{Unit-Discs}, 33) \ge 2$. On the other hand, unit balls in $\reals^3$ or higher dimensions are not cover-decomposable~\cite{MLP88,PTT09}.

Pach, Tardos and T\'oth~\cite{PTT09} obtain several negative results by embedding a non-cover-decomposable tree-based construction (used in \prettyref{sec:vc}) in geometric settings. In this way they prove that the following families are not cover-decomposable:
$\reals^2|\textsc{Axis-Aligned-Rectangles}$; $\reals^2|\textsc{Translates}(A)$ when $A$ is a quadrilateral that is not convex; and $\reals^2|\textsc{Strips}$ and its dual.
In contrast to the latter result, it is known that $\CC(\reals^2|\textsc{Axis-Aligned-Strips}, r) \ge \lceil r/2 \rceil$~\cite{AC+10}. Recently it was shown~\cite{KP11} that $\reals^3|\textsc{Translates}(\reals^3_+)$ is cover-decomposable, giving cover-decomposability of $\reals^2|\textsc{Homothets}(T)$ for any triangle $T$ and a new proof (c.f.~\cite{Ke07}) for $\reals^2|$\textsc{Bottomless-Axis-Aligned-Rectangles}; the former contrasts with the fact that $\reals^2|\textsc{Discs}$ is not cover-decomposable~\cite{PTT09}.

For $\reals^2|\textsc{Halfspaces}$ several results are known. Smorodinsky and Yuditsky~\cite{SY10} proved
$\CCDual(\delta) = \lceil \delta/2 \rceil$ (improving upon~\cite{BBSW10}) and $\CC(r) \ge \lceil r/3 \rceil.$ They note the latter is not tight, since Fulek~\cite{Fulek10} showed $\CC(3) = 2$.

Aside from the tree-based construction of \cite{PTT09}, we mention two other indecomposable constructions. First, the Hales-Jewett theorem is used in \cite{PTT09} to show that $\CC(\reals^2|\textsc{Lines}, r) = 1$ for all $r$.
Second, \cite{Pa09m} gives the following indecomposable construction, which is smaller than that of \cite{PTT09}. The ground set of the hypergraph is the set of all strings on at most $r-1$ B's and at most $r-1$ R's.
For every $c$ in $\mathbf{Z}_{\ge 0}^r$ with $\sum_i c[i] \le r-1$, there is a hyperedge $\{\MR^{c[1]}, \MR^{c[1]}\MB\MR^{c[2]}, \dotsc, \MR^{c[1]}\MB\MR^{c[2]}\MB\dotsb\MB\MR^{c[r]}\}$, and another hyperedge obtained by swapping the roles of B and R.
Then \cite{Pa09m} shows such hypergraphs have $\Chi=1.$


A recent paper of Chan et al.~\cite{CGKS12} centers around \emph{quasi-uniform} distributions over set covers, which are fractional packings of set covers. In contrast, $\cd$ deals with integral packing of set covers. The inspiration in \cite{CGKS12} is rooted in geometric settings, but they arrive at general combinatorial results.

For a hypergraph $H$ let $M_H$ be its 0-1 incidence matrix, with rows for edges and columns for vertices.
Then the following three properties are equivalent~\cite{Sc03}: (i) for every hypergraph obtained from $H$ by deleting or duplicating vertices (columns), $\Chi = r$; (ii) the fractional vertex cover polytope $\{x \mid x \ge 0, M_Hx \ge 1\}$ has the \emph{integer decomposition property}; (iii) the \emph{blocker} of $H$ is \emph{Mengerian}. A special case of such hypergraphs are \emph{balanced} hypergraphs, where no submatrix of $M_H$ is the incidence matrix of an odd cycle.


A generalization of polychromatic $c$-colouring to \emph{$c$-strong colouring} is considered in \cite{BWY12}: every hyperedge of size at least $c$ must receive all $c$ colours, and every smaller hyperedge must receive no colour twice. This can be further generalized to requiring an arbitrary lower bound $f(S)$ on the number of colours for every set $S$; when $f$ is the maximum of two supermodular functions this is the classical \emph{supermodular colouring} problem solved originally by Schrijver~\cite{Sc85}.

\section{Hypergraphs of Bounded Edge Size}\label{sec:bdsize}
To get good upper bounds on $\CCDual(\Hgr(\R), \delta)$, we will use the Lov\'asz Local Lemma (LLL):
\begin{lemma}[LLL, \cite{LLL}]Consider a collection of ``bad" events such that each one has probability at most $p$, and such that each bad event is independent of the other bad events except at most $D$ of them. (We call $D$ the \emph{dependence degree}.) If $p (D+1)\mathrm{e} \le 1$ then with positive probability, no bad events occur.
\end{lemma}

\noindent Our first tool extends a standard argument about Property B~\cite[Theorem 5.2.1]{AS08}.\begin{proposition}\label{proposition:lllusage}
$\CCDual(\Hgr(R), \delta) \ge \lfloor\delta/\ln(\mathrm{e}R\delta^2) \rfloor.$
\end{proposition}

\begin{proof}
Given any hypergraph $H = (V, \E)$ where every edge has size at most $R$ and such that each $v \in V$ is covered at least $\delta$ times, we must show for $t = \lfloor\delta/\ln(\mathrm{e}R\delta^2) \rfloor$ that $\ChiDual(H) \ge t$, i.e.\ that $\E$ can be decomposed into $t$ disjoint set covers. It will be helpful here and later to make the degree of every vertex \emph{exactly} $\delta$, since this bounds the dependence degree. (A similar observation was made by Mani-Levitska and Pach~\cite{MLP87} in the setting of cover-decomposing discs.) This is without loss of generality: otherwise whenever $\deg(v) > \delta$ shrink some $E \ni v$ to $E \bs \{v\}$ until $\deg(v)$ drops to $\delta$; then observe that undoing shrinking preserves the property of being a set cover.

Consider the following randomized experiment: for each hyperedge $E \in \E$, assign a random colour between $1$ and $t$ to $E$. If we can show that with positive probability, every vertex is incident with a hyperedge of each colour, then we will be done.

For each vertex $v$ define the \emph{bad event} $\Ev_v$ to be the event that $v$ is not incident with a hyperedge of each colour. The probability of $\Ev_v$ is at most $t(1-\frac{1}{t})^\delta,$ by using a union bound. The event $\Ev_v$ only depends on the colours of the hyperedges containing $v$; therefore the events $\Ev_v$ and $\Ev_{v'}$ are independent unless $v, v'$ are in a common hyperedge. In particular the dependence degree is at most $(R-1)\delta < R\delta$, since each edge containing $v$ contains at most $R-1$ other vertices. It follows by LLL that if
$$R\delta t(1-\tfrac{1}{t})^\delta \le 1/\mathrm{e},$$
then with positive probability, no bad events happen and we are done. We can verify that $t = \delta / \ln(\mathrm{e}R\delta^2)$ satisfies this bound.
\qquad\end{proof}

We will next show that the bound can be raised to $\Omega(\delta/\ln R)$. Intuitively, our strategy is the following. The lower bound on $\cd$ given by \prettyref{proposition:lllusage}, $\delta/\ln (R\delta),$ is already $\Omega(\delta/\ln R)$ in the special case that $\delta \le R^{O(1)}$. For hypergraphs where $\delta \gg R$ we will show that we can partition $\E$ into $m$ parts $\E = \biguplus_{i=1}^m \E_i$ so that $\delta(V, \E_i)$ is at least a constant fraction of $\delta/m$, and such that $\delta/m$ is polynomial in $\R$. Thus by \prettyref{proposition:lllusage} we can extract $\Omega((\delta/m)/\ln \R)$ set covers from each $(V, \E_i)$, and taken all together we obtain $\Omega(\delta/\ln \R)$ set covers.

In fact, it will be enough to consider splitting $\E$ into two parts at a time, recursively. Then ensuring $\delta(V, \E_i) \gtrsim \delta/2 ~ (i=1, 2)$ amounts to a dual discrepancy-theoretic problem: we must 2-colour the hyperedges by $\pm 1$ so that for each vertex, the sum of the incident hyperedges' colours is in $[-D, D]$, with the discrepancy $D$ as small as possible. To get a short proof of a weaker version of \prettyref{theorem:1}, we use a theorem of Beck and Fiala~\cite{BF81}. Later, in \prettyref{sec:structured}, we will extend their approach to trees.

\begin{theorem}[Beck \& Fiala~\cite{BF81}, stated in the dual]\label{theorem:bf}
In a $\delta$-regular hypergraph $H = (V, \E)$ with all edges of size at most $R$, we can partition the edge set into $\E = \E_1 \uplus \E_2$ such that $\delta(V, \E_i) \ge \delta/2 - R$ for each $i \in \{1, 2\}$.
\end{theorem}

Here is how the Beck-Fiala Theorem gives a near-optimal bound on $\CCDual$.
\begin{proposition}\label{proposition:fooo}
$\CCDual(\Hgr(R), \delta) \ge \Omega(\delta/\ln R).$
\end{proposition}
\begin{proof}
If $\delta < 4R$ this already follows from \prettyref{proposition:lllusage}. Otherwise apply \prettyref{theorem:bf} to the initial hypergraph, and then use shrinking to make both the resulting $(V, \E_i)$'s into regular hypergraphs. Iterate this process; stop splitting each hypergraph once its degree falls in the range $[R, 4R)$, which is possible since $\delta \ge 4R \Rightarrow \delta/2-R \ge R$. Let $M$ be the number of hypergraphs at the end.

Observe that in applying the splitting-and-shrinking operation to some $(V, \E)$ to get $(V, \E_1)$ and $(V, \E_2)$, the sum of the degrees of $(V, \E_1)$ and $(V, \E_2)$ is at least the degree of $(V, \E)$, minus $2R$ ``waste". It follows that the total waste is at most $2R(M-1)$, and we have that $4RM + 2R(M-1) \ge \delta$. Consequently $M \ge \delta/6R$. As sketched earlier, applying \prettyref{proposition:lllusage} to the individual hypergraphs, and combining these vertex covers, shows that $\ChiDual \ge M \lfloor R/\ln(\mathrm{e}R^3) \rfloor$ which gives the claimed bound.
\qquad\end{proof}

Now we get to the better bound with the correct multiplicative constant. We will need the following proposition, which when stated in the dual setting, says that the discrepancy of a hypergraph is at most $2 \sqrt{R \ln(R\Delta)}$.
\begin{proposition}\label{proposition:chernoff}
In a hypergraph $H = (V, \E)$ with maximum degree $\Delta$ and maximum edge size $R$, we can partition the edge set into $\E = \E_1 \uplus \E_2$ such that for every vertex $v,$ the degree of $v$ in each $\E_i$ is at least $\mathrm{degree}_H(v)/2 - \sqrt{\Delta \ln(R\Delta)},$ assuming $R\Delta > 1$.
\end{proposition}

\begin{proof}
Let $d_v$ be short for $\mathrm{degree}_H(v)/2.$ Independently assign each edge to $\E_1$ or $\E_2$ uniformly at random. For a $\lambda\ge0$ to be fixed later, for each vertex, let the event $\Ev_v$ be that the degree of $v$ is less than $\mathrm{degree}_H(v)/2 - \lambda$ in some $\E_i.$ Chernoff bounds \cite[Corollary A.1.2]{AS08} imply that $\Ev_v$ has probability at most $2\exp(-2\lambda^2/d_v) \le 2\exp(-2\lambda^2/\Delta)$.

If with positive probability no bad events happen, then we are done. As in the proof of \prettyref{proposition:lllusage}, the dependence degree is strictly less than $R\Delta$. Therefore the LLL is applicable provided that $2 \exp(-2\lambda^2/\Delta) R \Delta \le 1/\mathrm{e},$ which is satisfied by $\lambda = \sqrt{\Delta \ln(2 \mathrm{e} R\Delta)/2}$. This gives the desired result under the condition that $R\Delta \ge 2\mathrm{e}$. To finish the proof, one must check the cases $R=1<\Delta$, $\Delta=1<R$, and $\Delta=R=2$, which are straightforward.
\qquad\end{proof}

Now we get to the main proof.

{\emph{Proof of \prettyref{theorem:1}}. We want to show that for all hypergraphs with minimum degree $\delta$ and maximum hyperedge size $R$, that $\Chi \ge \delta/(\ln R+O(\ln \ln R)).$ Due to the crude bound in \prettyref{proposition:fooo}, we may assume that $R$ is larger than any chosen fixed constant. Moreover, \prettyref{proposition:lllusage} gives us the desired bound when $\delta$ is at most polylogarithmic in $R$, so we assume $\delta \ge \ln^K R$ for a positive constant $K$ to be fixed later. By shrinking, we assume the hypergraph is $\delta$-regular.

Let $d_0 = \delta$ and $d_{i+1} = d_i/2 - \sqrt{d_i\ln(Rd_i)}$. We mimic the split-and-shrink proof of \prettyref{proposition:fooo}, using \prettyref{proposition:chernoff} for the splitting. After $i$ rounds, all $2^i$ hypergraphs are regular with degree at least $d_i$. We stop splitting after $T$ rounds, where $T$ will be fixed later to make $d_T$ and $\delta/2^T$ polylogarithmic in $R$. Using the evident bound $d_i \le \delta/2^i$, the total degree loss due to splitting is at most
\begin{align} \delta - 2^T d_T &= \sum_{i=0}^{T-1} 2^i(d_i-2d_{i+1})= \sum_{i=0}^{T-1} 2^i2 \sqrt{d_i\ln Rd_i}
\notag  \\&\le \sum_{i=0}^{T-1} 2^i2 \sqrt{(\delta/2^i)\ln (R\delta/2^i)} = 2\sqrt{\delta}\sum_{i=0}^{T-1} \sqrt{2}^i \sqrt{ \ln (R\delta/2^i)}.
\label{eq:gs}
\end{align}
In this series, $\sqrt{2}^i$ is increasing in $i$, and $\sqrt{ \ln (R\delta/2^i)}$ decreases with $i$ but not too quickly;
the ratio between the $i$th and $(i+1)$th terms is
$$\sqrt{\ln(R\delta/2^i)}/\sqrt{\ln(R\delta/2^{i+1})} = \sqrt{1+\ln 2/\ln (R\delta/2^{i+1})} \le \sqrt{1+\ln_R 2}<1.3$$
where the first inequality uses $\delta/2^{i+1} \ge 1$ and the second uses $R \ge 3$.
Hence the sequence in \eqref{eq:gs} grows at least geometrically with ratio $\sqrt{2}/1.3$ and its sum is within a constant factor of the final $i=T-1$ term, using the geometric series sum formula.

We deduce $\delta - 2^T d_T = O(\sqrt{\delta 2^T \ln (R\delta/2^T)})$. Pick $K=3$, and pick $T$ such that $\delta/2^T$ is between $\ln^3R$ and $\ln^3R/2$, then we have
$$d_T \ge \delta/2^T\Bigl(1 - O\bigl(\sqrt{2^T\delta^{-1} \ln (R\delta/2^T)}\bigr)\Bigr) = \delta/2^T(1 - O(\ln^{-1}(R))).$$
Consequently with \prettyref{proposition:lllusage} we see that
$$\ChiDual \ge 2^T d_T/(\ln R + O(\ln \ln R)) \ge \delta(1 - O(\ln^{-1}(R)))/(\ln R + O(\ln \ln R))$$
which gives the claimed bound.
\qquad\endproof

\subsection{Sparse Hypergraphs: Proof of \prettyref{corollary:sparse}}\label{sec:sparse}
We re-use the following fact, a consequence of Hall's theorem, that was used in \cite{ABB+09}; see also~\cite[Thm.~2.4.2]{LP86}.
\begin{fact}
A hypergraph has at most $\alpha|V'|+\beta|\E'|$ incidences between each $V' \subseteq V$ and $\E' \subseteq \E$, if and only if the following assignment problem has a valid solution: for every $e$ and every $v \in e$, we need to assign the incidence $(v, e)$ to either $v$ or $e$, and the number of total incidences assigned to each vertex (resp.~edge) is at most $\alpha$ (resp.~$\beta$).
\end{fact}

To prove \prettyref{corollary:sparse}, we apply the fact, and then use shrinking to remove the incidences assigned to vertices. We retain minimum degree $\delta-\alpha$ and maximum edge size $\beta$. Finally, the corollary follows from \prettyref{theorem:1}.

\subsection{Lower Bounds}
Now we show that the bounds obtained previously are tight.

\emph{Proof of \prettyref{theorem:bounds}(i)}. We want to show that for some constant $C,$ all $R \ge 2$, and all $\delta \ge 1,$ we have $\CCDual(\Hgr(R), \delta) \le \max\{1, C\delta/\ln R\}$.
Since $\CCDual(\Hgr(R), \delta)$ is non-increasing in $R$, we may reduce $R$ by a constant factor to assume that  $R = \tbinom{2k-2}{k-1}$ for some integer $k \ge 2$. Note this gives $k =  \Theta(\ln R)$.

Consider the hypergraph $H = \tbinom{[2k-1]}{k}^*$ in the introduction. It is $k$-regular, it has $\ChiDual(H) = 1$, and $R(H) = \tbinom{2k-2}{k-1} = R$.
If $\delta \le k$ then $H$ proves the theorem for small enough $C$, so assume $\delta \ge k$. Again by monotonicity, we may increase $\delta$ by a constant factor to make $\delta$ a multiple of $k$. Let $\mu = \delta/k$.

Consider the hypergraph $\mu H$ obtained by copying each of its edges $\mu$ times, for an integer $\mu \ge 1;$ note that it is $\delta$-regular. The argument in the introduction shows that any set cover has size at least $k$ and therefore
average degree at least $k \tbinom{2k-2}{k-1} / \tbinom{2k-1}{k} = k^2/(2k-1) = \Theta(\ln R)$. Thus $\CCDual(\mu H) = O(\delta / \ln R)$ which proves the theorem.
\qquad\endproof

Next, \prettyref{theorem:bounds}(ii) establishes the right multiplicative constant in the range $\delta = \omega(\ln R)$.
Our approach relies on the methods for neighbourhood hypergraphs established in \cite{FHKS02}; in fact using \cite{FHKS02} as a black box gives \prettyref{theorem:bounds}(ii) for the case $R \sim \delta$. As part of reaching the fuller range, \prettyref{claim:journey} will generalize a calculation \cite[\S 2.5]{FHKS02} for dominating sets in Erd\H{o}s-R\'enyi graphs.

\emph{Proof of \prettyref{theorem:bounds}(ii)}. We must show, given a sequence $\{(R_i, \delta_i)\}_i$ such that $R \to \infty, \delta \to \infty,$ and $\delta = \omega(\ln R)$ as $i \to \infty$, we have $\CCDual(\Hgr(R), \delta) \le (1+o(1))\delta/\ln(R)$. We
assume an additional hypothesis, that $R \ge \delta$; this will be without loss of generality as we can handle the case $\delta > R$ using the $\mu$-replication trick from the proof of \prettyref{theorem:bounds}(i), since our argument is again based on lower-bounding the minimum size of a set cover.

Let $\delta' = \delta(1+o(1))$ and $R' = R(1-o(1))$ be parameters that will be specified shortly.
We construct a random hypergraph with $n = R'^2\delta'$ vertices and $m = R'\delta'^2$ edges, where for each vertex $v$ and each edge $E$, we have $v \in E$ with independent probability $p = 1/R'\delta'$. Thus each vertex has expected degree $\delta'$ and each edge has expected size $R'$. A standard Chernoff bound together with $np = \omega(\ln m)$ shows the maximum edge size is $(1+o(1))R'$ asymptotically almost surely (\emph{a.a.s.}); pick $R'$ such that this $(1+o(1))R'$ equals $R$. Likewise, since $mp = \omega(\ln n)$ we may pick $\delta' = (1+o(1))\delta$ so that a.a.s.\ the minimum degree is at least $\delta.$ Now we prove that small set covers are not likely to exist:
\begin{claim}\label{claim:journey}
A.a.s.\ the minimum  set cover size is at least $\frac{1}{p}\ln(pn)(1-o(1))$.
\end{claim}

\begin{proof}
Fix $\epsilon>0$ and let us a.a.s.~lower-bound the minimum set cover size by $\frac{1}{p}\ln(pn)(1-\epsilon)$.
For any fixed collection of $s = \frac{1}{p}\ln(pn)(1-\epsilon)$ edges,
let $P$ be the probability that it is a set cover. In the limit, $p$ tends to 0, so $(1-p)=\exp(-p+\Theta(p^2))$ and
\begin{align*}
P = (1-(1-p)^s)^n &= \bigl(1 - \exp(-s(p+\Theta(p^2)))\bigr)^n
\\&= (1 - \exp(-(1+\Theta(p))(1-\epsilon)\ln(pn)))^n
\\&= (1 - (pn)^{-1+\epsilon-\Theta(p)})^n.\end{align*}
Since $1-x \le \exp(-x)$ for all $x$,
$$ P \le \exp( - n(pn)^{-1+\epsilon-\Theta(p)}) = \exp( - p^{-1}(pn)^{\epsilon-\Theta(p)}).$$
There are $\tbinom{m}{s} \le m^s = \exp(s \ln m)$ such collections, so a.a.s.~none are set covers provided that
$$s \ln m- p^{-1}(pn)^{\epsilon-\Theta(p)} \to -\infty.$$
In turn, expanding the definition of $s$, we need precisely that
$$p^{-1}(\ln(pn)(1-\epsilon)\ln m- (pn)^{\epsilon-\Theta(p)}) \to -\infty.$$
To see this, note that $p = o(1)$, $pn=\omega(1)$, and $m$ is polynomial in $pn$ (since $R \ge \delta$).
\qquad\end{proof}

Using \prettyref{claim:journey} we can complete the proof of \prettyref{theorem:bounds}(ii), since the former implies that the maximum number of disjoint set covers $\ChiDual$ is at most $(1+o(1))mp/\ln(pn) = (1+o(1))\delta'/\ln(R') = (1+o(1))\delta/\ln(R).$
\qquad\endproof

Aside from the results above, not much else is known about specific values of $\CCDual(\Hgr(R), \delta)$ for small $R, \delta$. The Fano plane gives $\CC(\Hgr(3), 3) = 1$: if its seven sets are partitioned into two parts, one part has only three sets, and it is not hard to verify the only covers consisting of three sets are pencils through a point and therefore preclude the remaining sets from forming a cover. Moreover, Thomassen~\cite{Th92} showed that every 4-regular, 4-uniform hypergraph has Property B; together with monotonicity we deduce that $\CC(\Hgr(3), 4) \ge \CC(\Hgr(4), 4) \ge 2$.

\section{Paths in Trees}\label{sec:structured}
Let $\textsc{TreeEdges}|\textsc{Paths}$ denote the family of all hypergraphs that can be expressed in the following form: the ground set is the edge set $E_T$ of some undirected tree $T$, and for each hyperedge $E$ of $H$, there are vertices $u, v \in V_T$ so that $E$ is the set of all edges in $T$'s $u$-to-$v$ path.
Note that interval hypergraphs correspond to the special case that the tree is itself a path graph.

Our main result for these tree hypergraphs is that they are cover-decomposable, and moreover the cover-decomposition number is linear in $\delta$. To prove this theorem we use a counting lemma \cite[Lemma 4]{KPP08} for linear program extreme points. While \cite[Lemma 4]{KPP08} is stated for a linear program with just upper bounds on the capacity of edges, what we need is a result when there are both upper and lower bounds. However, the proof method is identical, as we will explain. 

We give a simpler sketch of the main idea just for intuition. We will write a linear program about two-colouring paths so as to achieve small dual discrepancy. For every path/hyperedge $P$ let $x_P$ and $y_P$ be non-negative variables constrained by $x_P + y_P = 1$. The idea is that $(x_P, y_P) = (1, 0)$ corresponds to placing $P$ in $\E_1$, and $(x_P, y_P) = (0, 1)$ corresponds to placing $P$ in $\E_2$. Consider the linear program
$$\mathbf{0} \le x, y; x+y = \mathbf{1}; \forall e, \sum_{P: e \in P} x_P = \textrm{degree}_H(e)/2;
\forall e, \sum_{P: e \in P} y_P = \textrm{degree}_H(e)/2.$$
If we could find an feasible integral solution $(x, y)$, then every vertex would be covered an equal number of times by $\E_1$ and $\E_2$, and the dual discrepancy would be 0. There is always a feasible fractional solution $x = y = \mathbf{\frac{1}{2}}$, but it is not guaranteed that a feasible integral solution exists. But, we can prove that there is an integer solution with $\pm O(1)$ additive violation in each constraint.
In the real proof, we will eliminate the $y$ variables according to $y = \mathbf{1}-x,$ and seek an unequal split in order to get a better factor in the final result.

\newcounter{oldthm}
\setcounter{oldthm}{\value{theorem}}
\setcounter{theorem}{\value{thm4}}

\begin{theorem}\label{theorem:kppplus}
$\CCDual(\textsc{TreeEdges}|\textsc{Paths}, \delta) \ge 1+\lfloor(\delta-1)/5\rfloor$.
\end{theorem}

\setcounter{theorem}{\value{oldthm}}

\begin{proof}
We use induction on $\delta.$ We assume that $\delta \ge 6$ since the smaller cases are trivial. Consider the following linear program (without objective).
\begin{align}
\mathbf{0} &\le x \le \mathbf{1} \tag{box}\label{eq:box} \\ \forall e \in E_T, \sum_{P: e \in P} x_P &\ge \ell_e := 3 \label{eq:lb}\tag{lb} \\
\forall e \in E_T, \sum_{P: e \in P} x_P &\le u_e := 3+\textrm{degree}_H(e)-\delta(H). \label{eq:ub}\tag{ub}
\end{align}
This linear program has at least one feasible solution: $x = \mathbf{\frac{3}{\delta}}$, which is easily seen to satisfy \eqref{eq:box} and \eqref{eq:lb}, and which satisfies \eqref{eq:ub}:
$$3\textrm{degree}_H(e)/\delta \le 3 + \textrm{degree}_H(e) - \delta \Longleftrightarrow 0 \le (\delta-3)(\textrm{degree}_H(e)-\delta)/\delta.$$
We run the following iterated relaxation algorithm, where the initial LP is \eqref{eq:box}--\eqref{eq:ub}.

\vspace{3mm}

\begin{center}
\noindent \begin{boxedminipage}{0.9\textwidth}
\noindent 1.\,\,Let $x^*$ be an extreme point solution of the current LP. \\
\noindent 2.\,\,If some path $P$ has $x^*_P$ equal to 0 or 1, then \emph{fix $x_P$}: from now on we consider each appearance of $x_P$ in the linear program to be a constant, equal to $x^*_P$. \\
\noindent 3.\,\,If for some edge $e$ there are only three non-fixed paths passing through $e$, then \emph{relax $e$}: delete the constraints \eqref{eq:lb} and \eqref{eq:ub} for $e$. \\
\noindent 4.\,\,Quit if all variables are fixed, otherwise return to the first step.
\end{boxedminipage}
\end{center}

\vspace{3mm}

The LP remains feasible from iteration to iteration since fixing and relaxing both preserve feasibility of $x^*$.
The main part of the proof will be to prove that the algorithm terminates. Assuming that this is so, here is how we complete the proof of \prettyref{theorem:kppplus}.

Let $x^f$ be the vector of final fixed values. We claim that every constraint \eqref{eq:lb} and \eqref{eq:ub} is almost satisfied by $x^f$, with at most an additive $\pm 2$ violation. This is trivially true for each $e$ that was never relaxed. If $e$ is relaxed in some iteration, at that time
$\sum_{P: e \in P} x^*_P \ge \ell_e$ holds, and the left-hand side is a sum of some fixed integers plus at most three non-fixed values strictly between 0 and 1; since $\ell_e$ is an integer, it follows that the fixed integers sum to at least $\ell_e-2.$ Hence these fixed values yield
$\sum_{P: e \in P} x^f_P \ge \ell_e - 2,$
no matter what happens to the non-fixed values over the remaining course of the algorithm. For \eqref{eq:ub} a similar argument establishes that
$\sum_{P: e \in P} x^f_P \le u_e + 2.$

By our choice of $\ell_e = 3$, the inequalities $\sum_{P: e \in P} x^f_P \ge \ell_e - 2$ ensures that $x^f$ is the indicator vector of a set cover $\E_1$. Similarly, $\sum_{P: e \in P} x^f_P \le u_e + 2$ ensures that $\mathbf{1} - x^f$ is the indicator vector of a family $\E_2$ of paths so that
the hypergraph $H' = (E_T, \E_2)$ has $\delta(H') \ge \delta(H) - 5.$ Apply \prettyref{theorem:kppplus} inductively to $H'$; these $1+\lfloor(\delta-6)/5\rfloor$ many set covers, together with $\E_1$, give the desired result.

To prove \prettyref{theorem:kppplus} it remains only to prove that the algorithm terminates. In order to do this it is helpful to write the modified LP explicitly, collecting all fixed constants. Define $\ell'_e$ to be $\ell_e$ minus the $x$ values of all fixed paths passing through $e$, and define $u'_e$ similarly. Define the tree $T'$ by contracting all relaxed edges, so $E_{T'}$ is the set of all non-relaxed edges. Let $\mathcal{P}'$ be the non-fixed paths, updated to take the contraction into account. The modified LP, whose variables are $x'_P$ for each non-fixed $P$, is
\begin{equation}\mathbf{0} \le x' \le \mathbf{1} \textrm{ and } \forall e \in E_{T'}, \ell'_e \le \sum_{P \in \mathcal{P}': e \in P} x'_P \le u'_e.\label{eq:newlp}\end{equation}
Notice that extreme points of this LP correspond to extreme points of the one in the algorithm, so we denote both by $x^*$. If the algorithm does not fix any paths in a given iteration, notice that we have the strict inequality $\mathbf{0} < x^* < \mathbf{1}$.
\begin{claim}\label{claim:term}
If $x^*$ is an extreme point solution to \eqref{eq:newlp} with $\mathbf{0} < x^* < \mathbf{1}$, then some $e \in E_{T'}$ lies on at most two paths from $\mathcal{P}'$.
\end{claim}

\begin{proof}
Since $x^*$ is an extreme point solution and no box constraints are tight, there are disjoint sets $E_\ell, E_u \subseteq E_{T'}$ so that $x^*$ is the unique solution to
$$\sum_{P: e \in P} x^*_P = \begin{cases} \ell'_e, & \forall e \in E_\ell;\\ u'_e, & \forall e \in E_u.\end{cases}$$
In particular, $|E_\ell|+|E_u| = |\mathcal{P}'|$ and the incidence matrix of $E_\ell \cup E_u$ versus $\mathcal{P}'$ is invertible. From this point on, the argument is the same as that at the top of page 9 in the proof of \cite[Lemma 4]{KPP08}. (The items $T=(V,E), E^*, D, c$ there correspond to $T', E_\ell \cup E_u, \mathcal{P}', \ell'/u'$ here. We make use of the integrality of $\ell'$ and $u',$ and linear independence in the aforementioned incidence matrix. )
\qquad\end{proof}

\prettyref{claim:term} proves that the algorithm either fixes or relaxes in each iteration, and therefore terminates.
So the proof of \prettyref{theorem:kppplus} is complete.
\qquad\end{proof}

This gives a 5-approximation algorithm for $\cd$ in this family of tree hypergraphs. From Holyer's result~\cite{Holyer81} mentioned earlier, it is \NP-hard to approximate $\cd$ within a factor less than 3/2 in the same family. We think \prettyref{theorem:2} is not tight; the best upper bound on $\CCDual$ we know is $\lfloor(3\delta+1)/4\rfloor$.

For \prettyref{theorem:kppplus} the ground set is the edge set of an undirected tree, but it is equally natural to consider two other settings. One is to let the ground set be the vertex set of an undirected tree, with hyperedges corresponding to the vertices in paths. Another is to let the ground set be the arc set of a bidirected tree, with hyperedges corresponding to the arcs in directed paths. In the forthcoming full version of \cite{KPP08} it is shown that \cite[Lemma 4]{KPP08} has analogues in these settings, with $\pm 6$ in place of $\pm 2$, and from this we get that $\cd \ge 1+\lfloor(\delta-1)/13\rfloor$ in these settings.

For polychromatic numbers and systems of paths in trees, we have:


\comment{As for the dual parameters, it is not hard to see that
$\PP(\textsc{TreeEdges}|\textsc{Paths}, 2) = +\infty$ by taking the tree to be a star with many leaves and the collection of all length-2 paths. (This is essentially the clique construction from~\cite{Pach80}). Finally, for completeness we state:}

\setcounter{oldthm}{\value{theorem}}
\setcounter{theorem}{\value{thm5}}

\begin{theorem}
$\CC(\textsc{TreeEdges}|\textsc{Paths}, r) = \lceil r/2 \rceil.$
\end{theorem}

\setcounter{theorem}{\value{oldthm}}

\begin{proof}
For the lower bound, colour the edges of the tree by giving all edges at level $i$ the colour $i \bmod \lceil r/2 \rceil$. Since every length-$r$ path in a tree has a monotonic subpath of length $\lceil r/2 \rceil$, we are done.

\comment{The lower bound means that for every collection of paths of length at least $2k-1$, we can $k$-colour the edges so that each path is polychromatic. To see this, root the tree arbitrarily and give all edges at level $i$ the colour $i \bmod k$. Every path uses at least $k$ consecutive levels, so we are done.}

\comment{
For the upper bound, it is enough to exhibit a tree and a family of $(2k-2)$-edge paths such that no $k$-colouring of the tree's edges is polychromatic with respect to the paths. We consider a complete $t$-ary tree of height $k-1$, for $t = 2k^{k-1}$, with root $r$, and where every path of length $2k-2$ (i.e.~each leaf-leaf path passing through $r$) is a hyperedge. Fix any colouring of the tree's edges. By the pigeonhole principle, there is a set $C_r$ of at least $t/k$ children of $r$, such that all edges $\{rc \mid c \in C_r\}$ get the same colour. Likewise, each node in $C_r$ exhibits the same phenomenon for some colour, so there is a subset $C'_r \subset C_r$ and node sets $\{C_c\}_{c \in C'_r}$ such that $C_c$ is a subset of $c$'s children of size at least $t/k$, such that all edges between $C'_r$ and $\bigcup_c C_c$ have a common colour, and such that $|C'_r| \ge t/k^2$. Continuing in this way, since $t/k^{k-1} \ge 2$, we find leaves $u, v$ such that the paths $ru$ and $rv$ contain the same ordered sequence of colours. Since the path has at most $k-1$ colours, it is not polychromatic, as needed.
}

For the upper bound, by monotonicity of $\CC$, it is enough to consider even $r = 2s$. Construct a spider tree that consists of a root vertex and $s+2$ paths of $s$ edges (each path has the root as one endpoint, and all other vertices are distinct). Define each of the $\tbinom{s+2}{2}$ leaf-leaf paths to be a hyperedge of the hypergraph; note that each one has size $r$. We must show that there is no polychromatic $(s+1)$-colouring. To see this, notice that each of the root-leaf paths must miss at least one colour, and by the pigeonhole principle two of them miss a common colour. The union of those two root-leaf paths is a hyperedge which is not coloured polychromatically, and we are done.
\qquad\end{proof}

\section{Small VC-Dimension}\label{sec:vc}
Recall the definition of VC-dimension: $S \subseteq V$ is \emph{shattered} by a hypergraph if every $T \subseteq S$ can be obtained as an intersection of $S$ with some hyperedge, i.e.~if $\{S \cap E \mid E \in \E\} = {\bf2}^S$; and the VC-dimension of a hypergraph equals the size of the largest shattered set. The \emph{dual VC-dimension} is the VC-dimension of the dual.

To show that primal and dual VC-dimension 2 is not enough to ensure cover-decomposability, it suffices to use the following construction of Pach, Tardos and T\'oth~\cite{PTT09}. Let $T$ be a $k$-ary rooted tree, with $k$ levels of vertices; so $T$ has $k^{k-1}$ leaves and $\sum_{i=0}^{k-1} k^i$ nodes in total. For each non-leaf node $v$ we define its \emph{sibling hyperedge} to be the $k$-set consisting of $v$'s children. For each leaf node $v$ we define its \emph{ancestor hyperedge} to be the $k$-set consisting of the nodes on the path from $v$ to the root node. We let $\PTT_k$ denote the hypergraph consisting of all ancestor and sibling hyperedges. Pach et al.~use a Ramsey-like argument to show $\Chi(\PTT_k) = 1$. Using this, we prove \prettyref{theorem:5}.

\setcounter{oldthm}{\value{theorem}}
\setcounter{theorem}{\value{thm7}}

\begin{theorem}
For the family of hypergraphs $\{H \mid \textrm{VC-dim}(H), \textrm{VC-dim}(H^*) \le 2\}$, we have $\CC(r) = 1$ for all $r$ and $\CCDual(\delta) = 1$ for all $\delta$.
\end{theorem}

\setcounter{theorem}{\value{oldthm}}

\begin{proof}
We will show that $\PTT_k$ has VC-dimension and dual VC-dimension at most 2. Thus $\PTT_k$ confirms the first part of \prettyref{theorem:5} (since $k=r$) and $\PTT_k^*$ confirms the second part. A key observation is that in $\PTT_k$,
\begin{equation}\textrm{any two distinct edges intersect in either 0 or 1 vertices.}\label{eq:PTT}\end{equation}

First, we bound the primal VC-dimension. Suppose for the sake of contradiction that there is a shattered vertex set $\{x, y, z\}$ of size 3. Since the set is shattered, there is a hyperedge $E$ containing all of $\{x, y, z\}$. But by the definition of shattering, there must be another hyperedge $E' \neq E$ with $E' \cap \{x, y, z\} = \{x, y\}$. This contradicts \eqref{eq:PTT}, so we are done.

Second, we bound the dual VC-dimension. Suppose for the sake of contradiction that $E, E', E''$ are three hyperedges which are shattered in the dual. This implies that $E \cap E' \cap E''$ and $(E \cap E') \bs E''$ are both nonempty. But this would imply $|E \cap E'| \ge 2$, contradicting \eqref{eq:PTT}.
\end{proof}

In the remainder of this section we show that set systems with unit VC-dimension have large cover-decomposition and polychromatic numbers (\prettyref{theorem:4}). The following is a convenient way of looking at such hypergraphs.

\begin{definition}
A hypergraph $(V, \E)$ is called \emph{cross-free} if the following holds for every pair $S, T \in \E$: at least one of $S \cap T, S \bs T, T \bs S,$ or $V \bs S \bs T$ is empty.
\end{definition}

Observe that a hypergraph has VC-dimension 1 if and only if the dual hypergraph is cross-free. In the proofs of this section we will also use \emph{laminar} hypergraphs, which are a subclass of cross-free hypergraphs.

\begin{definition}
A hypergraph $(V, \E)$ is called \emph{laminar} if the following holds for every pair $S, T \in \E$: at least one of $S \cap T, S \bs T$, or $T \bs S$ is empty. (Equivalently, either $S \cap T = \varnothing, S \subseteq T$, or $T \subseteq S$.)
\end{definition}

We need the following fact; a one-line proof is that such hypergraphs are balanced, but an elementary proof is also an easy exercise.
\begin{fact}\label{fact:laminarcoverdecomp}
In a laminar hypergraph, $\ChiDual = \delta$. \hfill 
\end{fact}

By duality, the following proposition proves the first half of \prettyref{theorem:4}.
\begin{proposition}\label{proposition:cf1}
For the family of cross-free hypergraphs, $\CCDual(\delta) = \lceil \delta/2 \rceil$.
\end{proposition}

\begin{proof}
First, for each $\delta$ we demonstrate a cross-free hypergraph with minimum degree $\delta$ which cannot be decomposed into more than $\lceil \delta/2 \rceil$ covers. Take a hypergraph whose ground set $V$ has $\delta+1$ elements, with one hyperedge $V \bs \{v\}$ for each $v \in V$. Since each set cover has at least 2 hyperedges, we have $\CCDual \le \lfloor (\delta+1)/2 \rfloor$, as needed.

To prove the other direction $\CCDual \ge \lfloor (\delta+1)/2 \rfloor$ for all cross-free hypergraphs, we use induction on $\delta$. Clearly the bound holds for $\delta=0$ or $\delta=1$. For the inductive step, we have two cases. First, if the cross-free hypergraph $(V, \E)$ has two sets $S, T \in \E$ for which $S \cup T = V$, then this is a cover with maximum degree 2 at each node. Thus we are done, since by induction $\CCDual(V, \E \bs \{S, T\}) \ge \lceil (\delta-2)/2 \rceil$. Second, if the cross-free hypergraph has no such $S, T$, then in fact the cross-free hypergraph is laminar, and we are done by \prettyref{fact:laminarcoverdecomp}.
\qquad\end{proof}

Again by duality, the following proposition proves the second half of \prettyref{theorem:4}.
\begin{proposition}\label{proposition:cf2}
For the family of cross-free hypergraphs, $\CC(r) = \lceil r/2 \rceil$.
\end{proposition}

\begin{proof}
Consider again the hypergraph used in the first half of the proof of \prettyref{proposition:cf1}. Note that it is self-dual, e.g.~since its incidence matrix is the all-ones matrix minus the identity matrix, which is symmetric. Thus, we deduce $\CC(r) \le \lfloor (r+1)/2 \rfloor$ for all $r$.

Next, we prove the other direction. We must show for all $k$, that if every hyperedge has size at least $2k-1$, then there is a polychromatic $k$-colouring. A key observation is the following: if $\mathcal{E}^\mathrm{min}$ denotes the inclusion-minimal elements of $\mathcal{E}$, then a colouring is polychromatic for $(V, \mathcal{E})$ if and only if it is polychromatic for $(V, \mathcal{E}^\mathrm{\min})$. So we may reset $\mathcal{E} := \mathcal{E}^\mathrm{\min}$, which does not affect $r$ or cross-freeness. Moreover now $\mathcal{E}$ is a \emph{clutter}: there do not exist two different hyperedges $A, B \in \mathcal{E}$ for which $A \subset B$.

\begin{lemma}\label{lemma:crossfreeclutter}
In a cross-free clutter $(V, \E)$, either all hyperedges are pairwise disjoint, or for every two hyperedges $A, B$ we have $A \cup B = V$.
\end{lemma}

\begin{proof}
From the definitions of ``cross-free" and ``clutter," we first have that for every $S, T \in \E$, either $S \cap T = \varnothing$ or $S \cup T = V$. Note that it is impossible for three distinct sets $S, T, U \in \E$ to satisfy $S \cup T = V, S \cap U = \varnothing$ since this would imply $U \subset T$. Using this transitively, we obtain the lemma.
\qquad\end{proof}

Using the lemma, \prettyref{proposition:cf2} now boils down to two cases. The first case is that all hyperedges are pairwise disjoint. In this case a polychromatic $k$-colouring is easy to obtain: for each hyperedge $E$, just colour its $|E| \ge 2k-1 \ge k$ elements in any way that uses all $k$ colours.

Finally, we deal with the case that for every two hyperedges $A, B$ we have $A \cup B = V$. Rewriting, this means that $\E$ is of the form $\{V \bs S_1, V \bs S_2, \dotsc, V \bs S_m\}$ where $m \ge 2$, the $S_i$ are pairwise disjoint, and $|V \bs S_i| \ge 2k-1$ for all $i$. Here is a polychromatic colouring method which picks one colour class at a time. Repeat the following for $j=1, \dotsc, k-1$: find any two vertices $v, v'$ such that no $S_i$ contains them both, and colour them with colour $j$. To see this is possible, note by induction that after iteration $j$, the number of uncoloured vertices in each $V \bs S_i$ is at least $2k-1-2j$. After stage $k-1$, colour all remaining uncoloured vertices with colour $k$.
\qquad\end{proof}

We remark that \prettyref{theorem:4} may be viewed from the perspective of a family of tree-related hypergraphs, distinct from the ones we previously studied. Namely, Edmonds \& Giles~\cite{EG77} showed that
a hypergraph $H$ is cross-free precisely when there exists a directed tree, such that each element of $V(H)$ occurs exactly once as a label on a node of the tree (a given node can have zero or multiple labels), and $\E(H)$ consists of those label sets on the half-tree ``pointed to" by each directed edge of the tree.

\section{$O(R)$-approximation for Sensor Cover}
The \emph{sensor cover} problem is a generalization of cover-decomposition. The input is a hypergraph plus, for each hyperedge $E$, a \emph{duration} $\ell_E$; a \emph{schedule} is an activation time $a_E$ for each hyperedge, and it has \emph{coverage} $T$ if for all $0 \le t \le T$ and all $v \in V$, some edge $E$ has $v \in E$ and $t \in [a_E, a_E + \ell_E]$. We are interested in finding a schedule with maximum coverage. Note that in the special case where all durations are unit, the maximum coverage equals the cover-decomposition number.

Let $\overline{\delta} = \min_v \sum_{E \ni v} \ell_E$, the minimum duration-weighted degree.
Clearly the maximum schedule coverage is at most $\overline{\delta}$. The most prominent results for this setting are from \cite{GK09} (the conference version): it is shown that every interval hypergraph has a schedule of coverage $\Omega(\overline{\delta})$; using this as a subroutine it is shown that for hypergraphs corresponding to translates of any fixed open convex polygon covering points in $\reals^2$, a schedule of $\Omega(\overline{\delta})$ always exists.

Let $R$ still mean the maximum number of vertices in any hyperedge, independent of the durations. We then obtain the following result.
\begin{theorem}\label{theorem:sensor-cover}
Every sensor cover instance admits a schedule of coverage $\Omega(\overline{\delta}/R)$.\end{theorem}
\begin{proof}
Without loss of generality, scale all durations uniformly so that $\overline{\delta}=1$. Let $\alpha$ be a parameter to be fixed later in the range $O(R)$, so that we will seek a schedule of coverage $1/\alpha$.

\comment{Clearly if there are any hyperedges of duration at least $1/\alpha$ we can schedule them at time zero, satisfying all their contained vertices; so without loss of generality all durations are $<1/\alpha$. Then, round down durations of the remaining hyperedges to the closest member of $\{\lambda/\alpha, \lambda^2/\alpha, \dotsc\}$ where $0<\lambda<1$ is another parameter to be fixed later. Let level $i$ consist of those hyperedges whose new durations are $\lambda^i/\alpha$, and let $d_i(v)$ be the degree of vertex $v$ in level $i$. Note that after rounding, the duration-weighted degree of each vertex is at least $\lambda$, so
\begin{equation}
\sum_{i \ge 1} \lambda^i d_i(v)/\alpha \ge \lambda.\label{eq:boundy}
\end{equation}
In each level we will use:}

Clearly if there are any hyperedges of duration at least $1/\alpha$ we can schedule them at time zero, satisfying all their contained vertices; so without loss of generality all durations are $<1/\alpha$. Then, round down durations of the remaining hyperedges to the closest member of $\{2^{-1}/\alpha, 2^{-2}/\alpha, \dotsc\}$. Let level $i$ consist of those hyperedges whose new durations are $2^{-i}/\alpha$, and let $d_i(v)$ be the degree of vertex $v$ in level $i$. Note that after rounding, the duration-weighted degree of each vertex is at least $1/2$, so
\begin{equation}
\sum_{i \ge 1}  \frac{d_i(v)}{2^i \alpha} \ge 1/2.\label{eq:boundy}
\end{equation}
In each level we will use:

\begin{claim}\label{claim:orient}
In a hypergraph, we can assign each hyperedge to a vertex it contains, so that for each $v$, the number of hyperedges assigned to it is at least $\lfloor d(v)/R \rfloor$.
\end{claim}
\begin{proof}
The naive LP formulation for this assignment problem (with integral data) is well-known to be integral. There is a feasible fractional assignment (fractionally assign each $E$ an amount $1/|E|$ to each vertex it contains), so there is also a feasible integral assignment.
\qquad\end{proof}

\comment{In each level, find the assignment specified by \prettyref{claim:orient}. We claim that, adding up all levels, the total duration assigned to each vertex is at least $1/\alpha$, provided we choose $\lambda$ correctly. If we can ensure this then we are done, since each hyperedge is assigned to just one vertex, and each vertex can dictate the scheduling of all hyperedges assigned to it. Since $\lfloor x/R \rfloor \ge (x-(R-1))/R$ for integral $x$ and $R$, the duration of the edges assigned to each $v$ is at least
\begin{align*}\sum_{i \ge 1} \lambda^i/\alpha \cdot (d_i(v)-(R-1))/R &= \frac{1}{R}\sum_{i \ge 1} \lambda^i d_i(v)/\alpha - \frac{R-1}{R}\sum_{i \ge 1} \lambda^i/\alpha \\&\ge \lambda/R - \frac{R-1}{R}\frac{\lambda}{(1-\lambda)\alpha},\end{align*}
where in the last inequality we used \eqref{eq:boundy} and the geometric series formula.
Provided the right-hand side above is at least $1/\alpha$, we are done. Solving gives $\alpha \ge (R-\lambda)/\lambda(1-\lambda).$ Thus we take the minimum $\alpha = 2R-1 + 2\sqrt{R^2-R}$, obtained at $\lambda = R-\sqrt{R^2-R}$. (Or as a rough approximation, $\lambda = 1/2, \alpha = 4R-2$.)}

In each level, find the assignment specified by \prettyref{claim:orient}. We will be done if we can pick $\alpha = O(R)$ so that, adding up all levels, the total duration assigned to each vertex is at least $1/\alpha$: a satisfactory schedule can then be obtained by, for each vertex, scheduling its assigned edges one after the other. Since $\lfloor x/R \rfloor \ge (x-(R-1))/R$ for integral $x$ and $R$, the duration of the edges assigned to each $v$ is at least
\begin{align*}\sum_{i \ge 1} \frac{\lfloor d_i(v)/R \rfloor}{2^{i} \alpha } \ge \sum_{i \ge 1} \frac{(d_i(v)-(R-1))/R}{2^{i} \alpha } &= \frac{1}{R}\sum_{i \ge 1}  \frac{d_i(v)}{2^i \alpha} - \frac{R-1}{R}\sum_{i \ge 1} \frac{1}{2^i \alpha} \\&\ge \frac{1}{R}\frac{1}{2} - \frac{R-1}{R}\frac{1}{\alpha},\end{align*}
where in the last inequality we used \eqref{eq:boundy} and the geometric series sum formula. We need precisely that the right-hand side above is at least $1/\alpha$. A short calculation reveals that $\alpha = 4R-2$ will do, and completes the proof.
\qquad\end{proof}

We remark that, by tweaking the scaling factor of $2$ in the proof above, the final ratio $\alpha$ can be slightly improved from $4R-2$ to $2R-1 + 2\sqrt{R^2-R}.$

\section{Open Problems}
The major open problem in our study is whether the sensor cover result \prettyref{theorem:sensor-cover} can be improved to a schedule of coverage $\Omega(\overline{\delta}/\ln R)$, in line with the tight results from \prettyref{sec:bdsize}. We also do not know if in $\textsc{TreeEdges}|\textsc{Paths}$, a schedule of coverage $\Omega(\overline{\delta})$ is always possible. For the non-scheduling version, it would be interesting to know more about $\CCDual(\Hgr(R), \delta)$ in the regime $\delta = \Theta(\ln R)$, or for small values of $\delta$ and $R$.

Two outstanding open geometric problems that predate our work are to determine whether $\CC(\delta) = \Omega(\delta)$ for the family of unit discs, and whether the family of all axis-aligned squares is cover-decomposable.

P{\'a}lv\"olgyi~\cite{Pa10} poses two nice combinatorial questions: first, is there a function $f$ so that in hypergraph families closed under edge deletion and duplication, $\CCDual(\delta_0) \ge 2$ implies $\CCDual(f(\delta_0)) \ge 3$? This compelling question, which would give a unified explanation for much of the existing literature, is open even for $\delta_0 = 2$; no counterexamples are known to the hypothesis $f(\delta_0) = O(\delta_0)$, i.e.~that cover-decomposability implies $\Chi$ grows linearly with $\delta$. Second, consider an $m \times r$ matrix of integers which is weakly increasing from left to right and top to bottom. If we think of each row as a set of size $r$, giving a hyperedge on ground set $\mathbb{Z}$, do these \emph{shift chains} have Property B for large enough $r$? It is known~\cite{Pa10} that $r = 3$ is not enough.

\subsubsection*{Acknowledgments}
We wish to thank Jessica McDonald, D\"om\"ot\"or P{\'a}lv\"olgyi, and Oliver Schaud for helpful discussions on these topics. We thank the ESA and SIDMA referees for their suggestions, in particular for the nicer proof of the second half of \prettyref{theorem:3}.

\bibliography{hc}

\bibliographystyle{siam}

\end{document}